\documentclass{aims_preprint}
\usepackage{amsmath}
\usepackage{paralist}
\usepackage{graphics}
\usepackage{epsfig}
\usepackage{graphicx}
\usepackage{epstopdf}
\usepackage[colorlinks=true]{hyperref}
\hypersetup{urlcolor=blue, citecolor=red}

\newtheorem{theorem}{Theorem}[section]

\newtheorem{lemma}[theorem]{Lemma}
\newtheorem{proposition}{Proposition}

\theoremstyle{definition}
\newtheorem{definition}[theorem]{Definition}
\newtheorem{remark}{Remark}

\newcommand{\e}{\varepsilon}

\newcommand{\argmin}[1]{\underset{#1}{\operatorname{argmin}}\;}

\title[Perturbed minimizing movements] 
      {Perturbed minimizing movements\\ of families of functionals}

\author[Andrea Braides and Antonio Tribuzio]{}

\subjclass{Primary:  47J30, 35K90, 49J45; Secondary: 47J35, 35B27.}
 \keywords{Gradient flows, variational evolution, $\Gamma$-convergence, homogenization, perturbations}

 \email{braides@mat.uniroma2.it}
 \email{tribuzio@mat.uniroma2.it}
 
\begin{document}
\maketitle

\centerline{\scshape Andrea Braides and Antonio Tribuzio}
\medskip
{\footnotesize
 \centerline{Department of Mathematics,
 University of Rome Tor Vergata}
   \centerline{via della Ricerca Scientifica, 00133 Rome, Italy}
}

\bigskip


\bigskip
 \centerline{\it Dedicated to Alexander Mielke on the occasion of his 60th birthday}

\begin{abstract}
We consider the well-known minimizing-movement approach to the 
definition of a solution of gradient-flow type equations by means of an implicit Euler scheme 
depending on an energy and a dissipation term. We perturb the energy by considering a 
($\Gamma$-converging) sequence and the dissipation by varying multiplicative terms.
The scheme depends on two small parameters $\e$ and $\tau$, governing energy and time scales, respectively.
We characterize the extreme cases when $\e/\tau$ and $\tau/\e$ converges to $0$ sufficiently fast,
and exhibit a sufficient condition that guarantees that the limit is indeed independent of $\e$ and $\tau$.
We give examples showing that this in general is not the case, and apply this approach to study some 
discrete approximations, the homogenization of wiggly energies and geometric crystalline flows obtained as limits of ferromagnetic energies.
\end{abstract}

\section{Introduction}
In the present paper we offer a contribution to the general problem of understanding the interaction between energy and dissipation terms in variational approaches to gradient-flow type evolutions from the standpoint of minimizing movements (see also e.g.~\cite{donfremie,fle,flesav} for related work).

Implicit Euler schemes are a well-established tool to prove existence and approximation for evolution equations 
with a gradient-flow structure. We follow De Giorgi's formalization \cite{deg}, which has allowed to use  such schemes as a basis for the definition and study of
gradient flows in metric spaces \cite{ambgigsav}. 
Given an initial datum $u^0$ and a functional $\phi$ defined on a metric space $(S,d)$, 
for fixed $\tau>0$ we denote by $\{u^\tau_n\}$ a discrete orbit satisfying $u^\tau_0=u^0$ and such that
$u^\tau_n$ be a minimizer of
\begin{equation}
	u\mapsto \phi(u)+\frac{d^2(u,u_{n-1}^{\tau})}{2\tau}.
\end{equation}
Any limit of a subsequence of the piecewise-constant interpolations $u^\tau(t)=u^\tau_{\lceil t/\tau\rceil}$ is called a {\em minimizing movement for $\phi$}. 
Such a limit exists under very mild conditions on $\phi$, and, under proper differentiability assumptions on $\phi$, is a curve of maximal slope for $\phi$, which is a generalization of the definition of a solution of the gradient-flow equation
\begin{equation}
	u'=-\nabla \phi(u)
\end{equation}
(see \cite{ambgigsav} Chapter 2).

In this paper we perturb the scheme above both considering a family of energies $\phi_\e$ depending on an additional parameter $\e$ in place of a single $\phi$, and a perturbation  by varying multiplicative coefficients $\{a_n^\tau\}$ of the squared-distance term (dissipation).
In this case the discrete orbits depend on $\e$ and $\tau$ and are defined by successive minimization requiring that $u^{\tau,\e}_n$ be a minimizer of
\begin{equation}\label{mingen}
	u\mapsto \phi_\e(u)+a_n^\tau\frac{d^2(u,u_{n-1}^{\tau,\varepsilon})}{2\tau}.
\end{equation}
By letting $\e$ and $\tau$ tend to $0$ at the same time, we then define the {\em$\{a^\tau\}$-perturbed minimizing movements along $\big\{\phi_\varepsilon\big\}$ at scale $\tau$} as all possible limits of subsequences of the corresponding piecewise-constant interpolations $u^{\tau,\e}(t)=u^{\tau,\e}_{\lceil t/\tau\rceil}$. 

In the case of a single function $\phi_\e=\phi$, the resulting {\em$\{a^\tau\}$-perturbed minimizing movements for $\phi$} have been analyzed in \cite{tri} showing on the one hand that, under proper local-summability assumptions on $\{1/a^\tau_n\}_n$, the resulting minimizing movements are a perturbed curve of maximal slope with rate $a^*$, which again extends the notion of a solution of the gradient-flow equation
\begin{equation}
	a^* u'=-\nabla \phi(u).
\end{equation}
Here, $1/a^*$ is a weak limit of the piecewise-constant interpolations of $\{1/a^\tau_n\}_n$. On the other hand, if the local-summability assumptions on $\{1/a^\tau_n\}_n$ fail, the resulting minimizing movement may result discontinuous and may be used to explore different energy wells. 

When varying energies $\phi_\e$ but no perturbation are considered (i.e, $a_n^\tau=1$ for all $\tau$ and $n$), the scheme above has been analyzed in \cite{bra2,bracolgobsol}, showing that in general the resulting {\em minimizing movements along $\big\{\phi_\varepsilon\big\}$ at scale $\tau$} do depend on how $\tau$ and $\e$ tend to $0$, even if we assume that $\phi_\e$ $\Gamma$-converge to some limit $\phi$ (which is not restrictive by a compactness argument). If equi-coerciveness assumptions on $\phi_\e$ hold, then diagonal arguments show that 
we may identify the limit motions in `fast-converging $\e$-$\tau$ regimes'. More precisely, if $\e$ converges sufficiently fast to $0$ with respect to $\tau$ then the limit is a minimizing movement for $\phi$, while if conversely $\tau$ converges sufficiently fast to $0$ with respect to $\e$ then it is a limit of minimizing movements for $\phi_\e$ as $\e\to 0$. It follows that, varying from fast-converging $\tau$ to fast-converging $\e$ we may encounter some $\tau(\e)$ (or $\e(\tau)$), which we call $\e$-$\tau$ {\it critical regimes}, for which the minimizing movements are different from those in the two fast-converging $\e$-$\tau$ regimes, and in general are in a sense an interpolation of the two extreme cases. All regimes give the same result if some general conditions envisaged by Colombo and Gobbino are satisfied by $\{\phi_\e\}$ \cite{bracolgobsol}, which in particular hold in the `trivial' case of convex energies but are forbidden by fast-oscillating energies. These conditions can be related to the previous seminal work by Sandier and Serfaty on limits of gradient flows \cite{SS}.

In the case of $\phi_\e$ $\Gamma$-converging to some limit $\phi$, the presence of many local minima may result in  a pinning phenomenon (i.e., orbits may be trapped in energy wells). The addition of the perturbations $\{a^\tau_n\}$ has the effect of allowing for a wider exploration of local energy wells, while maintaining a fixed overall effect on the limit continuum rate $a^*$. We prove that general $\{a^\tau\}$-perturbed minimizing movements along $\big\{\phi_\varepsilon\big\}$ interpolate between the fast-converging regimes given now by $\{a^\tau\}$-perturbed minimizing movements for $\phi$ and limits of minimizing movements for $\phi_\e$ as $\e\to 0$, and that the Colombo-Gobbino conditions still provide a `commutability result'. The effect in the critical regimes are examined in three sets of examples. First, we deal with one-dimensional discretizations of the simple energy $\phi(u)=-u$, showing that different perturbations with the same $a^*$ may give different pinning effects at the microscopic level, influencing the final homogenized velocity. The second example deals with one-dimensional wiggly energies, related to gradient flows of the type
\begin{equation}
	u'=-F'\Bigl({u\over\e}\Bigr),
\end{equation}
with oscillating $F$.
A minimizing-movement-based study of such energies has been performed in \cite{ansbrazim}, showing pinning phenomena in terms of the ratio $\e/\tau$. Here, we prove a general homogenization formula for the effective velocity, and an explicit description of the effect of the perturbations $\{a^\tau_n\}$ on the pinning threshold. Finally, a third example is given of a perturbed crystalline motion derived from lattice energies of ferromagnetic type as in \cite{bragelnov}, showing the dependence of the velocity and  consequently of the pinning threshold on the values of the perturbations.

\section{Perturbed minimizing movements along a family of functionals}\label{pertMM}

Following the notation in \cite{ambgigsav}, we consider a complete metric space $(S,d)$, and a Hausdorff topology $\sigma$ on $S$, weaker than the one induced by the metric $d$ and such that $d$ is $\sigma$-lower semicontinuous.

We consider a time-discretization parameter $\tau>0$. For every $\tau$ we consider a sequence $(a_n^\tau)_{n\ge1}$, such that $a_n^\tau>0$ for every $n\ge1$. We call this family of ($\tau$-parameterized) sequences  ``perturbations'', and we will use the notation
$$a^\tau:(0,+\infty)\to(0,+\infty),\quad a^\tau(t):=a_{\lceil t/\tau\rceil}^\tau$$
to denote the corresponding piecewise-constant interpolation, where $\lceil s\rceil$ denotes the upper integer part of $s$.

For each $\e>0$ we will consider a proper functional $\phi_\varepsilon:S\to(-\infty,+\infty]$ with the corresponding domains denoted by  $D(\phi_\varepsilon)$. 

With given $\e$ and $\tau$ we consider families of sequences $\{u_n^{\tau,\e}\}_n$ satisfying
\begin{equation}
\label{problem}
	\begin{cases}
		u_0^{\tau,\varepsilon}\in D(\phi_\varepsilon)\\
		u_n^{\tau,\varepsilon}\in\argmin{u\in S}\bigg\{\phi_\varepsilon(u)+a_n^\tau\frac{d^2(u,u_{n-1}^{\tau,\varepsilon})}{2\tau}\bigg\}& \hbox{ if }n\ge1.
	\end{cases}
\end{equation}
If such a family exists then we say that it {\em solves the Euler iterated minimization scheme along the sequence of functionals $\{\phi_\varepsilon\}$ at time discretization scale $\tau$ perturbed by $\{a^\tau\}$}. This scheme is a modified formulation of the one presented in \cite{ambgigsav}; in that case $\{a^\tau\}$ take the constant value $1$ and $\phi_\varepsilon$ are all equal to a single $\phi$. If, for a fixed $\tau$ and $\varepsilon$, there exist every step $u_n^{\tau,\varepsilon}$ of scheme (\ref{problem}), then the sequence is called a {\em discrete solution} or a {\em discrete orbit} for (\ref{problem}), and is identified with the curve
	$$u^{\tau,\varepsilon}:[0,+\infty)\to S,\quad u^{\tau,\varepsilon}(t):=u_{\lceil t/\tau\rceil}^{\tau,\varepsilon}.$$
	
We define gradient-flow type motions as the limits of $u^{\tau,\varepsilon}$ for $\e$ and $\tau$ tending to $0$.
The two parameters $\tau$ and $\varepsilon$ are thought to be related in the sense that they depend on each other, and the limit motion may depend on their relation. In order to highlight this, depending on the situation, we will write $\tau(\varepsilon)$ or $\varepsilon(\tau)$ and sometimes, we will refer to these relations with the term of ``$\tau$-$\varepsilon$ regimes''.

\begin{definition}
A curve $u:[0,+\infty)\to S$ is called a {\em$\{a^\tau\}$-perturbed minimizing movement along $\big\{\phi_\varepsilon\big\}$} if there exist two sequences $(\tau_k), (\varepsilon_k)$ both tending to $0$ as $k\to+\infty$ such that  discrete solutions $u^{\tau_k,\varepsilon_k}$ exist for every $k$ and pointwise converge to $u$ in the topology~$\sigma$.
\end{definition}

	For perturbations $\{a^\tau\}$ taking the constant value $1$, this definition is the same as that of minimizing movement along the sequence of functionals $\{\phi_\e\}$ at scale $\tau$ given in \cite{bracolgobsol}.
	For non-constant $\{a^\tau\}$ and a single functional this definition has been used in \cite{tri}. Reworking the arguments therein, which are themselves an elaboration of those in \cite{ambgigsav}, we have the properties contained in the following remark.

\begin{remark}[{\bf Assumptions for the existence of perturbed minimizing movements}]\label{epmm}
Following the case of a single (unperturbed) functional in \cite{ambgigsav}, we consider the following conditions:
\begin{enumerate}
	\item (lower semicontinuity) $\phi_\varepsilon$ are $\sigma$-lower semicontinuous for every $\varepsilon>0$
	\item (equicoerciveness) there exists $u^*\in S$ such that for all $c>0$ $$\inf_{u\in S, \varepsilon>0}\big\{\phi_\varepsilon(u)+c\, d^2(u,u^*)\big\}>-\infty$$
	\item (equicompactness) for all $c>0$ there exists a $\sigma$-compact $K_c$ such that $$\bigcup_{\varepsilon>0}\big\{u\in S\ \big\vert \ d(u,u^*)<c,\,\phi_\varepsilon(u)<c\big\}\subset K_c$$
	\item (control of initial data) there exists a constant $C_0$ such that for all $\tau,\,\varepsilon>0$, $d(u_0^{\tau,\varepsilon},u^*)\le C_0$ and $\phi_\varepsilon(u_0^{\tau,\varepsilon})\le C_0$
	\item (local uniform integrability) the family $\{1/a^\tau\}$ is uniformly integrable in $[0,T]$ for all $T>0$.
\end{enumerate}

These hypotheses imply (see \cite{tri}  Section 2 for details) that for any $T>0$ there exists a constant $C_T$ depending on the perturbations such that
\begin{equation}\label{precomp}
d(u^{\tau,\varepsilon}(t),u^*) \le C_T, \quad \phi_\varepsilon(u^{\tau,\varepsilon}(t)) \le C_0,\text{ for every }t\in[0,T];
\end{equation}
i.e., the $\sigma$-precompactness of discrete orbits, and a regularity of discrete solutions
\begin{equation}\label{regularity}
d(u^{\tau,\varepsilon}(t),u^{\tau,\varepsilon}(s)) \le c\,\theta_T(t+\tau,s),\,\hbox{ for all }t,s\in[0,T]
\end{equation}
where $c$ is a positive constant and
$$\theta_T(t,s)=\Big(\sup_{\tau>0} \int_s^t \frac{1}{a^\tau(\xi)}d\xi\Big)^{\frac{1}{2}}$$
defines a modulus of continuity. Applying a variant of the Ascoli-Arzel\'a Theorem (see Proposition 3.3.1 \cite{ambgigsav}) we obtain the existence of (at least) one perturbed minimizing movement $u\in AC_{\rm loc}(0,+\infty;S)$ as limit of a sequence $u^{\tau_k,\e_k}$.

Moreover, the increments of the discrete solutions
\begin{equation}\label{discder}
|(u^{\tau,\varepsilon})'|(t):=\frac{d(u_n^{\tau,\varepsilon},u_{n-1}^{\tau,\varepsilon})}{\tau},\text{ if }t\in((n-1)\tau,n\tau]
\end{equation}
weakly converge (up to subsequences) in $L^1_{\rm loc}(0,+\infty)$ to a function $A$, which is an upper bound for the {\em metric derivative} of $u$ (for its definition, see for instance \cite{ambgigsav} Theorem 1.1.2);  i.e.,
\begin{equation}\label{upbound}
|u'|(t):=\lim_{s\to t}\frac{d(u(t),u(s))}{|t-s|}\le A(t),\hbox{ a.e. in }(0,+\infty),
\end{equation}
and, defining (as in \cite{ambgigsav} Definition 3.2.1) for every $\tau, \varepsilon$ the {\em De Giorgi interpolants}
\begin{equation}\label{Gte}
\begin{aligned}
	&\tilde{u}^{\tau,\varepsilon}(t) \in \argmin{u\in S}\bigg\{\phi(u)+a^\tau_n\frac{d^2(u,u_{n-1}^\tau)}{2\delta}\bigg\} \\
	&G_{\tau,\varepsilon}(t) = a_n^\tau\frac{d(\tilde{u}^{\tau,\varepsilon}(t),u_{n-1}^\tau)}{\tau}
\end{aligned},\text{ if }t=(n-1)\tau+\delta
\end{equation}
we also obtain the following discrete energy estimate
\begin{equation}\label{eneest}
	\frac{1}{2}\int_0^{n\tau}a^\tau(\xi)|(u^{\tau,\varepsilon})'|^2(\xi)d\xi+\frac{1}{2}\int_0^{n\tau}\frac{1}{a^\tau(\xi)}G_{\tau,\varepsilon}^2(\xi)d\xi=\phi_\varepsilon(u_0^{\tau,\varepsilon})-\phi_\varepsilon(u_n^{\tau,\varepsilon}).
\end{equation}
for all $n\ge1$, that will bring to a convergence in energy.
\end{remark}

\begin{remark}[Curves of maximal slope with perturbed velocity]
	In \cite{ambgigsav} it is proved that minimizing movements for a single functional $\phi$ at the scale $\tau$ are curves of maximal slope with respect to $|\partial^-\phi|(u)$,  the {\em relaxed local slope} of $\phi$, which is defined as the $\sigma$-lower semicontinuous envelope of
	$$|\partial\phi|(u)=\liminf_{v\to u}\frac{\big(\phi(u)-\phi(v)\big)_+}{d(u,v)},$$
	under the assumption that $|\partial^-\phi|(u)$ be a {\em strong upper gradient} (see Definition 1.2.1 \cite{ambgigsav}). In the perturbed case, we need a generalization of the concept of curve of maximal slope for a functional in metric spaces to obtain the analogous result (Theorem 3.9 \cite{tri}) that $\{a^\tau\}$-perturbed minimizing movements for $\phi_\e$ are curves of maximal slope for some $\phi$ with a perturbed velocity.

\begin{definition}[Definition 3.2 \cite{tri}] Let $g:S\to[0,+\infty]$ be a strong upper gradient for $\phi$; that is, for any $v\in AC(a,c;S)$, $g\circ v$ is Borel and
$$|\phi(v(t))-\phi(v(s))|\le\int_s^t g(v(\xi))|v'|(\xi))d\xi,\hbox{ for any }a<s<t<b.$$
	If $\lambda:(a,b)\to(0,+\infty)$ is a measurable function, $u\in AC(a,b;S)$ is a {\em curve of maximal slope for $\phi$ with respect to a strong upper gradient $g$ of rate $\lambda$} if $\phi\circ u$ equals almost everywhere a non-increasing map (still denoted by $\phi\circ u$) and for all $a<s<t<b$
	\begin{equation}\label{cmp}
		\phi(u(t))-\phi(u(s))\le-\frac{1}{2}\int_t^s\frac{1}{\lambda(\xi)}|u'|^2(\xi)d\xi-\frac{1}{2}\int_t^s\lambda(\xi)g(u(\xi))^2d\xi.
	\end{equation}
	If $\lambda$ is constant then $u$ is a curve of maximal slope for $\phi$ with respect to $g$ according to the definition given in \cite{ambgigsav}.
\end{definition}
\end{remark}

\subsection{The conditions of Colombo-Gobbino}

As in \cite{bracolgobsol}, we prove that if the functionals converge in a strong way (conditions of Colombo-Gobbino below) we have a ``commutatibility result''.

\begin{definition}
	We say that a sequence of functionals $\big\{\phi_\varepsilon\big\}$ {\em converges to a functional $\phi$ according to the conditions of Colombo-Gobbino} if for every sequence $\varepsilon_k\to0$ and for all $v_k\xrightarrow{\sigma}v$, such that $\sup_k\big\{|\phi_{\varepsilon_k}(v_k)|,|\partial\phi_{\varepsilon_k}|(v_k)\big\}<+\infty$ we have
	\begin{equation}\label{colgob}
		\lim_k\phi_{\varepsilon_k}(v_k)=\phi(v),\quad\liminf_k|\partial\phi_{\varepsilon_k}|(v_k)\ge|\partial\phi|(v).
	\end{equation}
\end{definition}

With this condition we have the following result.

\begin{theorem}\label{CGthm}
	If assumptions from $1$ to $5$ of Remark {\rm\ref{epmm}} hold (so that there exists at least one perturbed minimizing movement), if a finite $a^*$ exists such that the functions $\{1/a^\tau\}$ weakly converge to $1/a^*$ in $L^1_{\rm loc}(0,+\infty)$ (which is always satisfied up to subsequences by assumption $5$ of Remark {\rm\ref{epmm}} ), if
	\begin{itemize}
		\item[{\rm (i)}] $\phi_\varepsilon$ converges to $\phi$ according to the conditions of Colombo-Gobbino
		\item[{\rm (ii)}] the local slope $|\partial\phi|$ is a strong upper gradient for $\phi$
	\end{itemize}
	then every $\{a^\tau\}$-perturbed minimizing movement along $\big\{\phi_\varepsilon\big\}$ of problem \eqref{problem} is a curve of maximal slope for $\phi$ with respect to $|\partial\phi|$ of rate $1/a^*$.
\end{theorem}
\begin{remark}
The assumption of finiteness on $a^*$ is only technical and can be avoided by a more precise definition of curve of maximal slope of given rate (as the one in \cite{tri}), that we omitted here for the sake of simplicity. Note moreover that $E=\{t\,|\,a^*(t)=+\infty\}$; that is, $1/a^\tau\rightharpoonup0$ on $E$, corresponds to the set of times in which the curve $u$ has zero velocity. 
\end{remark}
The proof of this theorem follows the one in \cite{ambgigsav} with the help of two additional technical results presented in \cite{tri} with all the details and recalled below. The first one can be proved using test functions lower than $\liminf_{\tau\to0}G_{\tau,\varepsilon}^2$ in the weak convergence of $1/a^\tau$. The second one can be obtained by slightly modifying the result of $\Gamma$-convergence of Dirichlet energy functionals on Sobolev spaces (Theorem 2.35 and Example 2.36 \cite{bra1}). 

\begin{lemma}\label{fatoul} Let $\tilde{u}^{\tau,\varepsilon}$ and $G_{\tau,\varepsilon}$ be defined as in \eqref{Gte}. Then 
	for every $t>0$ we have
	$$\liminf_{\tau,\varepsilon\to0}\int_0^{\lceil\frac{t}{\tau}\rceil\tau}\frac{1}{a^\tau(\xi)}G_{\tau,\varepsilon}^2(\xi)d\xi \ge \int_0^t\frac{1}{a^*(\xi)}\liminf_{\tau,\varepsilon\to0}G_{\tau,\varepsilon}^2(\xi)d\xi.$$
\end{lemma}
	
Since $|\partial\phi_\varepsilon|(\tilde{u}^{\tau,\varepsilon}(t)) \le G_{\tau,\varepsilon}(t)$ (Lemma 3.1.3 \cite{ambgigsav})	Lemma \ref{fatoul} implies
\begin{equation}\label{fatou}
	\liminf_{\tau,\varepsilon\to0}\int_0^{\lceil\frac{t}{\tau}\rceil\tau}\frac{1}{a^\tau(\xi)}G_{\tau,\varepsilon}^2(\xi)d\xi \ge \int_0^t\frac{1}{a^*(\xi)}\liminf_{\tau,\varepsilon\to0}|\partial\phi_\varepsilon|^2(\tilde{u}^{\tau,\varepsilon}(\xi))d\xi.
\end{equation}

\begin{lemma}\label{omol}
	Let $(u^{\tau,\e})'$ be defined as in \eqref{discder} and $A$ as in \eqref{upbound}. Let $\tau_k,\e_k$ be sequences such that $u^{\tau_k,\e_k}, (u^{\tau,\e})'$ and $a^{\tau_k}$ converge respectively to $u, A$ and $a^*$, then exists a subsequence (not relabeled) such that for every $t>0$
	\begin{equation}\label{omo}
		\liminf_{k}\int_0^{\lceil\frac{t}{\tau_k}\rceil\tau_k} a^{\tau_k}(\xi)|(u^{\tau_k,\varepsilon_k})'|^2(\xi)d\xi \ge \int_0^t a^*(\xi)A^2(\xi)d\xi.
	\end{equation}
\end{lemma}

\begin{proof}[Proof of Theorem {\rm\ref{CGthm}}]
	Taking the limit in the energy estimate (\ref{eneest}), thanks to the conditions of Colombo-Gobbino, and inequalities (\ref{fatou}) and (\ref{omo}), we have
	\begin{eqnarray*}
		\phi(u(0))&=&\lim_k\phi_{\varepsilon_k}(u_0^{\tau_k,\varepsilon_k})\\
		&\ge&\liminf_k\frac{1}{2}\int_0^{n\tau_k}a^{\tau_k}(\xi)|(u^{\tau,\varepsilon})'|^2(\xi)d\xi+\frac{1}{2}\int_0^{n\tau_k}\frac{1}{a^{\tau_k}(\xi)}G_{\tau_k,\varepsilon_k}^2(\xi)d\xi\\
		&&\quad\quad+\phi_{\varepsilon_k}(u^{\tau_k,\varepsilon_k}(t))\\
		&\ge&\frac{1}{2}\int_0^ta^*(\xi)A^2(\xi)d\xi+\frac{1}{2}\int_0^t\frac{1}{a^*(\xi)}\liminf_k|\partial\phi_{\varepsilon_k}|^2(\tilde{u}^{\tau_k,\varepsilon_k}(\xi))d\xi\\
		&&\quad\quad+\phi(u(t))\\
		&\ge&\frac{1}{2}\int_0^ta^*(\xi)|u'|^2(\xi)d\xi+\frac{1}{2}\int_0^t\frac{1}{a^*(\xi)}|\partial\phi|^2(u(\xi))d\xi+\phi(u(t))
	\end{eqnarray*}
	and the result follows.
\end{proof}

\subsection{Fast-converging sequences}

Now, we treat the case when we only make the assumption of $\Gamma$-convergence of the sequence of the functionals $\phi_\e$, which always holds up to subsequences in separable metric spaces. In what follows we set $\phi=\Gamma$-$\lim_\varepsilon\phi_\varepsilon$.
Under these weaker hypotheses not in every $\tau$-$\varepsilon$ regime do we have commutation between the $\Gamma$-limit and the minimizing movement, as shown by the following result, which is a readjustment to the perturbed case of the result in \cite{bra2} (Theorem 8.1).

\begin{theorem}\label{fastconv}
	Let conditions from $1$ to $5$ of Remark $\ref{epmm}$ hold, and let the family of functionals $\big\{\phi_\varepsilon\big\}$ be equi-mildly-coercive; that is, for every $c>0$ there exists a $d$-compact set $K_c$ such that for all $\varepsilon>0$
	$$\inf_{u\in S}\big\{\phi_\varepsilon(u)+c\,d^2(u,u^*)\big\}=\inf_{u\in K_c}\big\{\phi_\varepsilon(u)+c\,d^2(u,u^*)\big\}$$
	where $u^*$ is the same  as in condition {\rm 2} of Remark {\rm\ref{epmm}}. Then, if $u_0^{\tau,\e}=u_0^\e$ we have that
	\begin{itemize}
		\item[\rm(i)] there exists a scale $\varepsilon(\tau)$ such that if $\varepsilon\le\varepsilon(\tau)$ every $\{a^\tau\}$-perturbed minimizing movement along $\big\{\phi_\varepsilon\big\}$ is a $\{a^\tau\}$-perturbed minimizing movement with respect to $\phi$;
		\item[\rm(ii)] there exists a scale $\tau(\varepsilon)$ such that if $\tau\le\tau(\varepsilon)$ every $\{a^\tau\}$-perturbed minimizing movement along $\big\{\phi_\varepsilon\big\}$ is a limit curve of the sequence $\big\{u^\varepsilon\big\}$, where, for every fixed $\e$, $u^\varepsilon$ is a $\{a^\tau\}$-perturbed minimizing movement with respect to the single functional $\phi_\varepsilon$.
	\end{itemize}
\end{theorem}
\begin{proof}
	(i) by assumption 4 of Remark \ref{epmm} it is not restrictive to assume that $u_0^\e\to u_0$.
	With fixed $\tau$, for any sequence $\varepsilon_n\to0$ and any $v_n\to v$ the term $d^2(v_n,u_0^{\tau,\varepsilon_n})$ converges to $d^2(v,u_0)$ which implies that
			$$\Gamma\text{-}\lim_{\varepsilon\to0}\bigg(\phi_\varepsilon(u)+a_1^\tau\frac{d^2(u,u_0^{\tau,\varepsilon})}{2\tau}\bigg) = \phi(u)+a_1^\tau\frac{d^2(u,u_0)}{2\tau}$$
			and, by the equicoerciveness assumption, we have the convergence of the minima
			$$\lim_{\varepsilon\to0}\min_{u\in S}\bigg\{\phi_\varepsilon(u)+a_1^\tau\frac{d^2(u,u_0^{\tau,\varepsilon})}{2\tau}\bigg\} = \min_{u\in S}\bigg\{\phi(u)+a_1^\tau\frac{d^2(u,u_0^\tau)}{2\tau}\bigg\}.$$
			Hence, every minimizer $u_1^{\tau,\varepsilon}$ converges to a minimizer $u_1^\tau$ for the corresponding minimum problem with respect to $\phi$ when $\varepsilon$ tends to $0$. Repeating the same argument every $u_n^{\tau,\varepsilon}$ converges to the corresponding $u_n^\tau$ for every $n\ge1$ and so we have the convergence of the discrete solutions
			$$\lim_{\varepsilon\to0}u^{\tau,\varepsilon}=u^\tau.$$
			Since $u^\tau$ converges to a $\{a^\tau\}$-perturbed minimizing movement with respect to $\phi$, a diagonal argument defines $\varepsilon(\tau)$.
			
		(ii) with fixed $\varepsilon>0$, we have convergence of discrete solutions to $u^{\varepsilon}$ and these perturbed minimizing movements are equicompact and equicontinuous. This follows by passing to the limit in (\ref{precomp}) and (\ref{regularity}) in Remark \ref{epmm}, since these properties depend on the perturbations, which do not depend on $\varepsilon$. Hence, the result follows by the Ascoli-Arzel\`a Theorem.
\end{proof}

\begin{remark}\label{rmk-fastconv}
	In the sequel, we will use the notation $u^0$ and $u^\infty$ to indicate perturbed minimizing movements obtained under condition (i) and (ii), respectively, of  Theorem \ref{fastconv}.
\end{remark}

\section{Examples of critical regimes} From Theorem \ref{fastconv} we infer that, when the evolutions given by opposite types of fast-converging sequences differ, varying between those we may encounter one or more {\em critical $\e$-$\tau$ regimes}, at which the minimizing movements describe an effective motion different from the extreme ones. Various types of critical regimes have been already studied in \cite{bra2} in the case of un-perturbed minimizing movements. Here we highlight some effects of the perturbations with two simple examples.

\smallskip
On the real line, we consider the functions 
$$\phi_\varepsilon(t)=
\begin{cases}
	-t&t\in\varepsilon\mathbb{Z}\\
	+\infty&\text{otherwise}
\end{cases}$$
as a prototype of multi-well energies with different well depth. It is not restrictive to take
$u_0^{\tau,\varepsilon}\equiv0$. The perturbations $\{a^\tau\}$ are assumed to satisfy assumption 5 of Remark \ref{epmm}. 

All the hypotheses of Remark \ref{epmm} are satisfied, so that $\{a^\tau\}$-perturbed minimizing movements along $\big\{\phi_\varepsilon\big\}$ are well-defined curves $u:[0,+\infty)\to\mathbb{R}$. 
In this case, the iterated minimization algorithm \eqref{problem} takes the form
$$
u_n^{\tau,\varepsilon}\in\argmin{u\in\varepsilon\mathbb{Z}}\bigg\{-u+a_n^\tau\frac{|u-u_{n-1}^{\tau,\varepsilon}|^2}{2\tau}\bigg\},
$$
so that $u_n^{\tau,\varepsilon}$ is the point of $\varepsilon\mathbb{Z}$  closest to the minimum of the parabola; that is, $u_{n-1}^{\tau,\varepsilon}+\tau/a_n^\tau$.

Note that if $\tau/a_n^\tau\in\varepsilon(\mathbb{Z}+1/2)$ there are two such points, so that we have
$$u_n^{\tau,\varepsilon}=
\begin{cases}
	u_{n-1}^{\tau,\varepsilon}+\tau/a_n^\tau\pm\varepsilon/2&\text{if }\tau/a_n^\tau\in\varepsilon(\mathbb{Z}+1/2)\\
	u_{n-1}^{\tau,\varepsilon}+\varepsilon\lfloor\tau/(\varepsilon a_n^\tau)+1/2\rfloor&\text{otherwise.}
\end{cases}$$
The cases of double minimizers can be treated separately, so for simplicity we consider the assumption
\begin{equation}\label{nob}
	\frac{1}{a_n^\tau}\not\in\frac{\varepsilon}{\tau}\bigg(\mathbb{Z}+\frac{1}{2}\bigg)
\end{equation}
in which case $\{u_n^{\tau,\varepsilon}\}$ is defined iteratively by
\begin{equation}\label{ste}
	u_n^{\tau,\varepsilon}=u_{n-1}^{\tau,\varepsilon}+\varepsilon\bigg\lfloor\frac{\tau}{\varepsilon a_n^\tau}+\frac{1}{2}\bigg\rfloor.
\end{equation}

\subsection{Pinning}
We say that the perturbed motion is {\em pinned} if there exists a constant $c>0$ such that for any $a^\tau>c$ then $u\equiv u_0$. We define the infimum of such constants (which may also be $0$) as the {\em pinning threshold} of the motion.

Define $\gamma=\gamma(\tau,\varepsilon)=\varepsilon/\tau$. Condition (\ref{nob}), which ensures the uniqueness of the minima, is
$$\frac{1}{a_n^\tau}\not\in\gamma\bigg(\mathbb{Z}+\frac{1}{2}\bigg).$$
By (\ref{ste}), if $a_n^\tau<2/\gamma$ then $u_n^{\tau,\varepsilon}>u_{n-1}^{\tau,\varepsilon}$, otherwise we have $u_n^{\tau,\varepsilon}=u_{n-1}^{\tau,\varepsilon}$. Hence, if $a^\tau>2/\gamma$ the motion is pinned; i.e., $u(t)\equiv0$ and $2/\gamma$ is the pinning threshold.
 
Note that $a^\tau>2/\gamma$ is a sufficient condition in order to obtain a pinned motion, but not necessary. In fact, consider the set $I_{\gamma,\tau}(t):=\{\xi\in[0,t]\,|\,a^\tau(\xi)\le2/\gamma\}$. By \eqref{ste} the discrete solution is
$$u^{\tau,\varepsilon}(t)=\varepsilon\sum_{n\tau\in I_{\gamma,\tau}(t)}\left\lfloor\frac{\tau}{\varepsilon a_n^{\tau}}+\frac{1}{2}\right\rfloor=\gamma\int_{I_{\gamma,\tau}(t)}\bigg\lfloor\frac{1}{a^{\tau}(\xi)\gamma}+\frac{1}{2}\bigg\rfloor d\xi$$
and we obtain the estimate
$$\big|I_{\gamma,\tau}(t)\big|\le u^{\tau,\varepsilon}(t)\le\int_{I_{\gamma,\tau}(t)}\frac{1}{a^{\tau}(\xi)}d\xi+\frac{\gamma}{2}\big|I_{\gamma,\tau}(t)\big|.$$
Hence, if the following condition over the perturbations $\{a^\tau\}$ is satisfied
\begin{equation}\label{pinning}
	\lim_k\big|I_{\gamma_k,\tau_k}(t)\big|=0 \hbox{ for all } t\ge0,
\end{equation}
where $\gamma_k=\gamma(\e_k,\tau_k)$, we have a pinned motion $u=\lim_k u^{\tau_k,\varepsilon_k}$. Otherwise, if for some $t_0\ge0$
$$\limsup_k\big|I_{\gamma_k,\tau_k}(t_0)\big|>0$$
taking the limit along a suitable subsequence of $(\tau_k)$ we obtain $u(t_0)>0$.

\begin{remark}
In case of $N$-periodic perturbations, we have that condition (\ref{pinning}) is satisfied if and only if
\begin{equation}\label{pinningth}
	\alpha:=\inf_{1\le n\le N}a_n^\tau>\frac{2}{\gamma}.
\end{equation}
So in this case, the pinned perturbed motions are characterized by the pinning threshold; 
i.e., if $\alpha>2/\gamma$ the motion is pinned, otherwise it is not.
\end{remark}

\subsection{Fast-convergences}\label{fastconvergences}

In this case, the scales defined in (i) and (ii) of Theorem \ref{fastconv}, can be chosen as $\tau(\varepsilon)=o(\varepsilon)$ and $\varepsilon(\tau)=o(\tau)$.

Indeed, consider $\varepsilon(\tau)=o(\tau)$, by (\ref{ste}) we have, for every $t\ge0$
	$$u^{\tau,\varepsilon}(t)=\sum_{n=1}^{\lceil t/\tau\rceil}\varepsilon\bigg\lfloor\frac{\tau}{\varepsilon a_n^{\tau}}+\frac{1}{2}\bigg\rfloor$$
so taking the limit for $\tau\to0$ in
	$$\sum_{n=1}^{\lceil t/\tau\rceil}\frac{\tau}{a_n^\tau}-\bigg\lfloor\frac{t}{\tau}\bigg\rfloor\varepsilon\le u^{\tau,\varepsilon}(t)\le\sum_{n=1}^{\lceil t/\tau\rceil}\frac{\tau}{a_n^\tau}+\bigg\lfloor\frac{t}{\tau}\bigg\rfloor\varepsilon$$
we obtain
	\begin{equation}\label{fast-motion}
	u^0(t)=\int_0^t\frac{1}{a^*(\xi)}d\xi,
	\end{equation}
a $\{a^\tau\}$-perturbed minimizing movement with respect to $\phi(t)=-t=\Gamma$-$\lim_{\varepsilon\to0}\phi_\varepsilon(t)$.
		
Now, let $\tau(\varepsilon)=o(\varepsilon)$, so that $1/\gamma(\varepsilon)\to0$ for $\varepsilon\to0$. Assumption 5 of Remark \ref{epmm} implies that
	$$\big|I_{\delta,\tau}(t)\big|\le\frac{2}{\gamma}\int_{I_{\delta,\tau}(t)}\frac{1}{a^\tau(\xi)}d\xi\le\frac{2}{\gamma}\bigg\Vert\frac{1}{a^\tau}\bigg\Vert_{L^1(0,t)}\le\frac{2}{\gamma}C_{0,t}.$$
Hence, pinning condition (\ref{pinning}) is satisfied; i.e., $u^\infty(t)\equiv0$ for these regimes, but for every $\varepsilon$ the perturbed minimizing movements $u^\varepsilon$ are identically $0$ because $\phi_\varepsilon$ has a discrete domain, so the result follows.

This shows that the critical regimes are such that $\varepsilon(\tau)=\gamma(\tau)\tau$, with $\gamma(\tau)$ a bounded function with $\inf_\tau\gamma(\tau)>0$. Without loss of generality consider the regimes
$$\varepsilon=\gamma\tau.$$
In what follows, we use the notation $u^\gamma$ for the $\{a^\tau\}$-perturbed minimizing movements along $\{\phi_{\gamma\tau}\}$.

\subsection{Periodic perturbations}

Now, given $0<\alpha<\beta$, choose general periodic perturbations
$$a_n^\tau=
	\begin{cases}
		\alpha&n\text{ odd}\\
		\beta&n\text{ even.}
	\end{cases}$$
Such perturbations *weakly converge to the inverse of the harmonic mean between $\alpha$ and $\beta$; that is, $1/a^*=(1/\alpha+1/\beta)/2$. Hence, from \eqref{fast-motion} and the analysis of the pinning phenomenon performed above we have
$$u^0(t)=\frac{1}{a^*}t,\quad u^\infty(t)\equiv0,$$
where we have used the notation introduced in Remark \ref{rmk-fastconv}.

In the critical regimes, we have different perturbed minimizing movements depending on $\gamma$, chosen according to condition \eqref{nob}. Define $k_\alpha=\lfloor1/(\alpha\gamma)+1/2\rfloor$ and $k_\beta=\lfloor1/(\beta\gamma)+1/2\rfloor$. By \eqref{ste} we have
$$u_n^{\tau,\varepsilon}=
	\begin{cases}
		u_{n-1}^{\tau,\varepsilon}+k_\alpha\varepsilon&n\text{ odd}\\
		u_{n-1}^{\tau,\varepsilon}+k_\beta\varepsilon&n\text{ even}
	\end{cases}
	=\bigg\lceil\frac{n}{2}\bigg\rceil k_\alpha\varepsilon+\bigg\lfloor\frac{n}{2}\bigg\rfloor k_\beta\varepsilon$$
and so the discrete solution is
$$u^{\tau,\varepsilon}(t)=\bigg\lceil\frac{t}{2\tau}\bigg\rceil k_\alpha\varepsilon+\bigg\lfloor\frac{t}{2\tau}\bigg\rfloor k_\beta\varepsilon=\bigg\lceil\frac{t}{2\tau}\bigg\rceil k_\alpha\gamma\tau+\bigg\lfloor\frac{t}{2\tau}\bigg\rfloor k_\beta\gamma\tau$$
and taking the limit we have
$$u^\gamma(t)=\frac{1}{a_\gamma}t,\text{ with }\frac{1}{a_\gamma}:=\gamma\frac{k_\alpha+k_\beta}{2}.$$

\begin{figure}[htp]
	\includegraphics[height=6 cm]{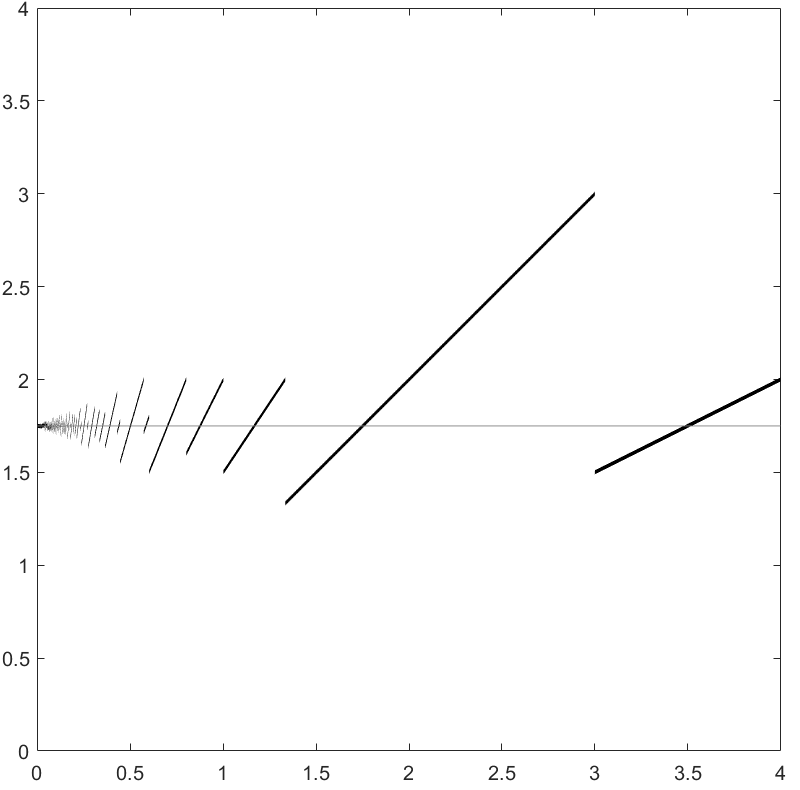}\includegraphics[height=6 cm]{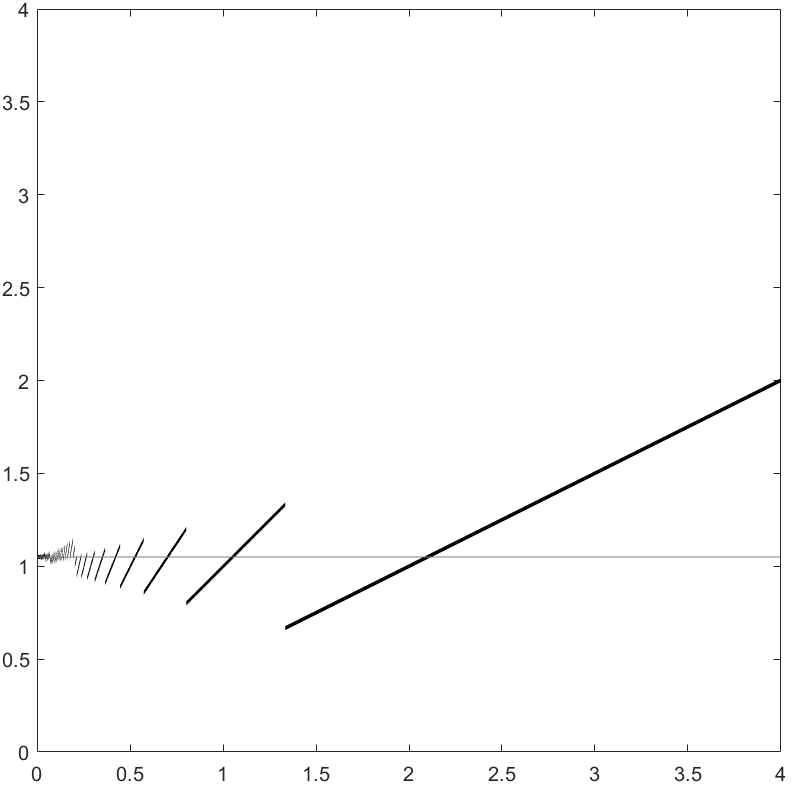}
	\caption[]{The dark line represents the graph of $\gamma\mapsto1/a^\gamma$, the light line is the constant $1/a^*$. On the left $\alpha>\beta/2$ so the sup is reached in $\gamma_1^\beta$, on the right $\alpha<\beta/2$ and the sup is reached in $\gamma_1^\alpha$.}
	\label{fig1}
\end{figure}

\begin{remark}
By varying $\gamma$ we obtain different perturbed minimizing movements depending on the value $1/a_\gamma$. The function $\gamma\mapsto 1/a_\gamma$ is a piecewise linear function having jumps in the bifurcation values; i.e., by (\ref{nob}) in
$$\gamma^\alpha_j:=\frac{2}{(2j-1)\alpha},\quad\gamma^\beta_j:=\frac{2}{(2j-1)\beta},$$
and $1/a_\gamma=0$ for all $\gamma>\gamma_1^\alpha$ as stated by condition (\ref{pinningth}). Moreover $1/a_\gamma\to1/a^*$ when $\gamma\to0$ (see Figure \ref{fig1}). In fact $1/a_\gamma$ can be seen as an approximation of the harmonic mean between $\alpha$ and $\beta$ taking values on $\gamma\mathbb{Z}$.

We shortly examine the determination of the largest velocity that the motion could reach. To do this it suffices to evaluate $1/a^\gamma$ at the right-end extremes of the continuity intervals, since there the function is increasing. This corresponds to considering the greatest jump value at $\gamma_j^\alpha$ and $\gamma_j^\beta$ for all $j\ge1$. These are
\begin{align*}
\frac{1}{a^{\gamma_j^\alpha}} &= \frac{j}{(2j-1)\alpha}+\frac{\gamma_j^\alpha}{2}\bigg\lfloor\frac{1}{\beta\gamma_j^\alpha}+\frac{1}{2}\bigg\rfloor=\frac{1}{(2j-1)\alpha}\left(j+\left\lfloor\frac{(2j-1)\alpha}{2\beta}+\frac{1}{2}\right\rfloor\right) \\
\frac{1}{a^{\gamma_j^\beta}} &= \frac{\gamma_j^\beta}{2}\bigg\lfloor\frac{1}{\alpha\gamma_j^\beta}+\frac{1}{2}\bigg\rfloor+\frac{j}{(2j-1)\beta}=\frac{1}{(2j-1)\beta}\left(\left\lfloor \frac{(2j-1)\beta}{2\alpha}+\frac{1}{2} \right\rfloor+j\right).
\end{align*}
Since $\alpha<\beta$, the value $(2j-1)\alpha/(2\beta)+1/2$ is less than $j$ so that its lower integer part is less then or equal to $j-1$, which yields $1/a^{\gamma_j^\alpha}\le1/\alpha=1/a^{\gamma_1^\alpha}$. While, when $\alpha\ge\beta/2$, $(2j-1)\beta/(2\alpha)+1/2\le2j-1/2$, so $1/a^{\gamma_j^\beta}\le1/\beta+j/((2j-1)\beta)<2/\beta=1/a^{\gamma_1^\beta}$. Now, since $1/\alpha\ge1/a^{\gamma_1^\alpha}$ if and only if $\alpha\le\beta/2$, we have that
$$\sup_{\gamma>0}\frac{1}{a^\gamma}=\max_{j\ge1}\left\{\frac{1}{a^{\gamma_j^\alpha}},\frac{1}{a^{\gamma_j^\beta}}\right\}=\begin{cases}
1/\alpha&\alpha\le\beta/2 \\
2/\beta&\alpha\ge\beta/2.\end{cases}$$
Therefore, the largest velocity of $u^\gamma$, which is reached in $2/\alpha$ if $\alpha\le\beta/2$ and in $2/\beta$ otherwise, depends on the ratio $\alpha/\beta$ (Figure \ref{fig1}).
\end{remark}

The behavior of $1/a^\gamma$ when $\gamma\to0,+\infty$ gives a compatibility property for the perturbed minimizing movements. Indeed
$$\lim_{\gamma\to0}u^\gamma=u^0,\quad\lim_{\gamma\to+\infty}u^\gamma=u^\infty,$$
uniformly on compacts subsets of $(0,+\infty)$.

\subsection{Different flows generated by perturbations with the same harmonic mean}

Slightly modifying the functionals, the situation could be more complicated. Consider the sets $Z_\varepsilon:=3\varepsilon\mathbb{Z}\cup\e(3\mathbb{Z}+1)$ and the functionals
$$\phi_\varepsilon(t)=\begin{cases}-t&t\in Z_\varepsilon\\+\infty&\text{otherwise}\end{cases}$$
with initial data $u_0^{\tau,\varepsilon}\equiv0$, and perturbations satisfying assumption 5 of Remark \ref{epmm}. As well as in the previous section, when $\tau=o(\varepsilon)$ or $\varepsilon=o(\tau)$ we have fast-convergences, so we study the case $\varepsilon=\gamma\tau$, with $\gamma\in(0,+\infty)$.

The $n$-th step of the discrete solution is the projection of $u_{n-1}^{\tau,\varepsilon}+\tau/a_n^\tau$ on $Z_\varepsilon$. Note that the points of the domain are not equidistanced. We define the projection on $Z_\varepsilon$
\begin{equation}\label{ste2}
	P_{Z_\e}(t)=\begin{cases}
t\pm\e/2 &\text{if }t\in\e(3\mathbb{Z}+1/2) \\
t\pm\e &\text{if }t\in\e(3\mathbb{Z}+2) \\
\e\lceil(2t+2)/(3\e)\rceil &\text{otherwise}\end{cases}
\end{equation}
and $u_n^{\tau,\varepsilon}=P_{Z_\varepsilon}(u_{n-1}^{\tau,\varepsilon}+\tau/a_n^\tau)$. Note that by (\ref{ste2}), condition (\ref{nob}), which ensures the absence of bifurcations, is replaced by
$$\frac{1}{a_n^\tau}\not\in\gamma\left(\left(\mathbb{Z}+\frac{1}{2}\right)\cup\left(\mathbb{Z}+2\right)\right)$$
and there are two critical values of the perturbations which affect the motion:
\begin{itemize}
	\item[(i)] if $a_n^\tau>2/\gamma$ \emph{total pinning}; i.e., $u_n^{\tau,\e}=u_{n+1}^{\tau,\e}$;
	\item[(ii)] if $1/\gamma<a_n^\tau<2/\gamma$ \emph{partial pinning}; i.e., if $u_{n-1}^{\tau,\e}+\e\not\in Z_\e$ then $u_n^{\tau,\varepsilon}=u_{n-1}^{\tau,\varepsilon}$ otherwise $u_n^{\tau,\varepsilon}=u_{n-1}^{\tau,\varepsilon}+\e$
	\item[(iii)] if $a_n^\tau<1/\gamma$ then $u_n^{\tau,\varepsilon}>u_{n-1}^{\tau,\varepsilon}.$
\end{itemize}

Consider $\gamma\in(1/2,1)$. We present two perturbations oscillating between the values 1 and 2 with a different period, having the same harmonic mean, which generate two different motions. According to the above conditions, when $a_n^\tau=2$ we are in the case we have partial pinning, when $a_n^\tau=1$ we have no pinning. Consider first
$$a_n^\tau=
\begin{cases}
	1&n\text{ odd}\\
	2&n\text{ even.}
\end{cases}$$
Except for an initial phase displacement, when $n$ is odd we have that $u_n^{\tau,\varepsilon}=u_{n-1}^{\tau,\varepsilon}+2\varepsilon$, and when $n$ is even we have $u_n^{\tau,\varepsilon}=u_{n-1}^{\tau,\varepsilon}+\varepsilon$.
Hence, for large $n$ we have
$$u_n^{\tau,\varepsilon}=o(1)+\left\lceil\frac{n}{2}\right\rceil\varepsilon+\left\lfloor\frac{n}{2}\right\rfloor2\varepsilon$$
and the discrete solution is
$$u^{\tau,\varepsilon}=o(1)+\gamma\left(\left\lceil\frac{1}{2}\left\lceil\frac{t}{\tau}\right\rceil\right\rceil\tau+\left\lfloor\frac{1}{2}\left\lceil\frac{t}{\tau}\right\rceil\right\rfloor2\tau\right).$$
By taking the limit as $\tau\to0$, we have
$$u^\gamma(t)=\gamma\frac{3}{2}t.$$

Now, consider the perturbations
$$a_n^\tau=
\begin{cases}
	1&n\in\big(4\mathbb{N}+1\big)\cup(4\mathbb{N}+2)\\
	2&n\in\big(4\mathbb{N}+3\big)\cup4\mathbb{N}.
\end{cases}$$
After an initial phase displacement, we have the following periodic situation. Let $n\in4\mathbb{N}+1$, we always have that $u_n^{\tau,\e}=\e(3k+1)$ for some integer $k$ and
\begin{itemize}
	\item[(i)] $a_n^\tau=1$ we have no pinning and $u_{n+1}^{\tau,\e}=u_n^{\tau,\e}+2\e$;
	\item[(ii)] $a_{n+1}^\tau=1$ again no pinning and $u_{n+2}^{\tau,\e}=u_n^{\tau,\e}+3\e$;
	\item[(iii)] $a_{n+2}^\tau=2$ we have partial pinning and $u_{n+2}^{\tau,\e}+\e\not\in Z_\e$, so $u_{n+3}^{\tau,\e}=u_{n+2}^{\tau,\e}$;
	\item[(iv)] $a_{n+3}^\tau=2$ and, as above, $u_{n+4}^{\tau,\e}=u_{n+2}^{\tau,\e}.$
\end{itemize}
Hence, for any large $n$ we have $u_{n+4}^{\tau\varepsilon}=u_n^{\tau,\varepsilon}+3\varepsilon$; that is, $u_{4n}^{\tau,\varepsilon}=o(1)+4\varepsilon+3n\varepsilon$, and the discrete solution is
$$u^{\tau,\varepsilon}(t)=o(1)+\gamma\left(\left\lceil\frac{1}{4}\left\lceil\frac{t}{\tau}\right\rceil\right\rceil\tau+\left\lfloor\frac{1}{4}\left\lceil\frac{t}{\tau}\right\rceil\right\rfloor2\tau\right).$$
Taking the limit we obtain
$$u^\gamma(t)=\gamma\frac{3}{4}t.$$
	
\section{Oscillating energies}

In this section we study the homogenization of perturbed gradient flows along wiggly energies, which has already been treated in its unperturbed formulation in \cite{ansbrazim} and in \cite{bra2} (Example 8.2). Consider a positive constant $T>0$ and $W\in C^1(\mathbb{R})$ a $1$-periodic even function with Lipschitz derivative, $\Vert W'\Vert_\infty=1$ and zero average. Now, consider the energies
$$\phi_\e(t)=\e\, W\bigg(\frac{t}{\e}\bigg)+Tt,$$
with initial data $u_0^{\tau,\e}=u_0^\e$, and assume that all the hypotheses of Remark \ref{epmm} are satisfied, so that we have perturbed gradient flows along $\phi_\e$.

In \cite{ansbrazim}, it is proved that critical regimes for this motion are those such that the ratio $\varepsilon/\tau$ is bounded with positive infimum, and the case $\varepsilon=\gamma\tau$ is studied. In such critical regimes, there exists a pinning threshold of initial data and minimizing movements are linear functions with a homogenized velocity. We will prove that for perturbed minimizing movements analogous results hold.

\subsection{Fast convergences}

We prove that in the regimes $\varepsilon(\tau)=o(\tau)$ and $\tau(\varepsilon)=o(\varepsilon)$ we have fast convergences.

Consider $\e(\tau)=o(\tau)$. Denote by $\phi(u)=Tu$ the $\Gamma$-limit of $\phi_\e$. In order to lighten the notation, for every $n$ and $\tau$ we define $F(u):=\phi(u)+a_n^\tau(u-u_{n-1}^{\tau,\e})^2/2\tau$ and $F_\e(u):=\phi_\e(u)+a_n^\tau(u-u_{n-1}^{\tau,\e})^2/2\tau$. By a direct computation we can write $F_\e(u)=F(u)+\e\, W(u/\e)$. Now, denote
	$$v:=\argmin{u\in\mathbb{R}}F(u)=u_{n-1}^{\tau,\varepsilon}-\frac{\tau}{a_n^\tau}T.$$
Since $\big|\varepsilon W(t/\varepsilon)|\le\varepsilon/2$, we have that
\begin{equation}\label{cond}
	F(u_n^{\tau,\varepsilon})<F(v)+\varepsilon.
\end{equation}
	Indeed, otherwise $F_\e(u_n^{\tau,\e}) \ge F(u_n^{\tau,\e})-\e/2 \ge F(v)+\e/2 > F_\e(v)$, which is in contrast with the minimality of $u_n^{\tau,\varepsilon}$. By the minimality of $v$ we have that $F(u) = a_n^\tau(u-v)^2/2\tau+F(v)$, so that, by \eqref{cond}, we have that
\begin{equation}\label{dist}
	|u_n^{\tau,\e}-v| \le \left(\frac{2\tau\e}{a_n^\tau}\right)^{1\over 2}.
\end{equation}
Now take $\{n\},\{m\}$ two family of integers (depending on $\tau$) such that $n>m$ and both $n\tau$ and $m\tau$ converge to $t\ge0$. By \eqref{dist}, applying a discrete H\"older inequality we have that
	$$\left|u_n^{\tau,\varepsilon}-u_m^{\tau,\varepsilon}+T\sum_{i=m+1}^{n}\frac{\tau}{a_i^\tau}\right| \le \sum_{i=m+1}^{n}\left(2\frac{1}{a_i^\tau}\tau\varepsilon\right)^{1\over2} \le \bigg(2\tau\varepsilon(n-m)\sum_{i=m+1}^n\frac{1}{a_i^\tau}\bigg)^{\frac{1}{2}}$$
	and dividing both the members by $(n-m)\tau$ we get
	$$\bigg|\frac{u_n^{\tau,\varepsilon}-u_m^{\tau,\varepsilon}}{(n-m)\tau}+T\frac{1}{n-m}\sum_{i=m+1}^{n}\frac{1}{a_i^\tau}\bigg| \le \bigg(2\frac{\varepsilon}{\tau}\frac{1}{n-m}\sum_{i=m+1}^n\frac{1}{a_i^\tau}\bigg)^{\frac{1}{2}}.$$
	Taking the limit as $\tau\to0$, by Lebesgue's Theorem (up to subsequences) we have that
	$$u'(t)=-\frac{1}{a^*(t)}T,\text{ for almost every } t\ge0;$$
	hence, $u$ is a $\{a^\tau\}$-perturbed minimizing movement with respect to $\phi$.	
	
In order to show that the other extreme regime is $\tau(\varepsilon)=o(\varepsilon)$, we first characterize $u^\infty$.
For the single functional $\phi_\varepsilon$, perturbed minimizing movements are the solutions to
\begin{equation}\label{ode}
	u^\varepsilon(t)'=-\frac{1}{a^*(t)}\bigg(T+W'\bigg(\frac{u^\varepsilon(t)}{\varepsilon}\bigg)\bigg) \hbox{ for all }t\ge0.
\end{equation}
First, note that, for $T\le1$, the set of constant solutions $\big\{x\in\mathbb{R}\,\big|\,T+W'(x/\e)=0\big\}$ tends to be dense, so that for every initial value $u_0^\varepsilon$ (converging to some $u_0$) the solution $u^\varepsilon$ lives in an interval of length $\varepsilon$, so that $u^\infty(t)\equiv u_0$ for any $T\le1$.
Whereas, for $T>1$ integrating (\ref{ode}) from $t_1$ to $t_2$ we obtain
	$$\int_{u^\varepsilon(t_1)}^{u^\varepsilon(t_2)} \frac{1}{T+W'(s/\e)}ds=-\int_{t_1}^{t_2}\frac{1}{a^*(t)}dt.$$
	Now, $1/(T+W'(s))$ is a summable $1$-periodic function, so that the integrand on the left-hand side $L^1$-weak converges to the average $\int_0^1 1/(T+W'(s))ds$. We define the function
	$$f(T)=\begin{cases}
	0 & \text{if }T\le1\\ \displaystyle
	\Big(\int_0^1 {1\over T+W'(s)}ds\Big)^{-1} & \text{if }T>1.
	\end{cases}$$
	Taking the limit for $\varepsilon\to0$, and $t_1\to0, t_2\to t$ we have
	$$u^\infty(t)=u_0-f(T)\int_0^t\frac{1}{a^*(\xi)}d\xi \qquad\hbox{ for all } t\ge0.$$
If $\tau(\varepsilon)=o(\varepsilon)$ the same argument applied at the time-discrete level shows that $u^{\tau,\e}\to u^\infty$ as $\e\to0$.

\subsection{Critical regimes}

We now study the critical regimes. It is not restrictive to suppose that $\varepsilon=\gamma\tau$, with $\gamma>0$. Since
$$\argmin{u\in\mathbb{R}}\bigg\{\phi_\varepsilon(u)+a_n^\tau\frac{(u-u_{n-1}^{\tau,\varepsilon})^2}{2\tau}\bigg\}=\argmin{u\in\mathbb{R}}\bigg\{\frac{1}{\varepsilon}\phi_\varepsilon(u)+\frac{1}{\varepsilon}a_n^\tau\frac{(u-u_{n-1}^{\tau,\varepsilon})^2}{2\tau}\bigg\},$$
we reduce to the following rescaled problem by taking $y=u/\e$
\begin{equation}\label{rescaledp}
	\begin{cases}
		y_0^\tau=y_0\in\mathbb{R},\\ \displaystyle
		y_n^\tau\in\argmin{y\in\mathbb{R}}\Big\{\phi_1(y)+{a_n^\tau\over 2}\gamma(y-y_{n-1}^\tau)^2\Big\}.
	\end{cases}
\end{equation}
We obtain that $u_n^{\tau,\varepsilon}=\varepsilon y_n^\tau$, provided that $y_0=u_0^\varepsilon/\varepsilon$. The rescaled discrete solutions $y^\tau$ in general depends on $\tau$ because of the perturbations. If one takes, as in the following, periodic perturbations the dependence on $\tau$ will disappear.

First, we recall a useful result presented and proved in \cite{ansbrazim} Proposition 3.1.

\begin{lemma}\label{mp}
Given any function $\psi_1,\psi_2:\mathbb{R}\to\mathbb{R}$, and a positive constant $\beta>0$. For any $x_1,\,x_2\in\mathbb{R}$, if
$$y_1\in\argmin{t\in\mathbb{R}}\{\psi_1(t)+\beta(t-x_1)^2\},\quad y_2\in\argmin{t\in\mathbb{R}}\{\psi_2(t)+\beta(t-x_2)^2\}$$
then $\psi_1(y_1)-\psi_1(y_2)+\psi_2(y_2)-\psi_2(y_1)\le2\beta(x_1-x_2)(y_1-y_2)$.
\end{lemma}

{
By mimicking the argument of the proof of the previous lemma, we obtain the following useful result.

\begin{lemma}\label{pinninglem}
Given $\psi:\mathbb{R}\to\mathbb{R}$, and two positive constants $0<\alpha<\beta$. For any $y_0\in\mathbb{R}$, if
$$y^\alpha\in\argmin{t\in\mathbb{R}}\{\psi(t)+\alpha(t-y_0)^2\},\quad y^\beta\in\argmin{t\in\mathbb{R}}\{\psi(t)+\beta(t-y_0)^2\}$$
then $(y^\alpha-y^\beta)$ has the same sign of $(y^\alpha-y_0)$.
\end{lemma}
\begin{proof}
By the minimality of $y^\alpha$ and $y^\beta$ we can write
\begin{align*}
\psi(y^\alpha)+\alpha(y^\alpha-y_0)^2 &\le \psi(y^\beta)+\alpha(y^\beta-y_0)^2 \\
\psi(y^\beta)+\beta(y^\beta-y_0)^2 &\le \psi(y^\alpha)+\beta(y^\alpha-y_0)^2
\end{align*}
and summing them up we get
$$(\alpha-\beta)(y^\alpha-y_0)^2\le(\alpha-\beta)(y_1^\beta-y_0)^2$$
which implies, on one hand that $|y^\alpha-y_0|\ge|y^\beta-y_0|$, on the other hand
$$0\le(y^\alpha-y^\beta)((y^\beta-y_0)+(y^\alpha-y_0)).$$
Since $(y^\beta-y_0)+(y^\alpha-y_0)$ has the sign of $y^\alpha-y_0$, we get the thesis.
\end{proof}
}

The monotonicity of unperturbed discrete solutions follows from  Lemma \ref{mp}, as shown in  \cite{ansbrazim}. In the perturbed case we have no monotonicity of the discrete solutions. Nevertheless, an important property remains; namely, the monotone behavior with respect to initial data. This is ensured by the following result  (see \cite{ansbrazim} Proposition 4.2), which still holds in the perturbed case with the same proof.

\begin{proposition}\label{monbeh}
Given two initial data $y_0,z_0\in\mathbb{R}$, and two positive constants $S\ge T>0$. Consider $\{y_n^\tau\}$ solutions of \eqref{rescaledp} starting from $y_0$ with $\phi_1(t)=W(t)+Tt$ and $\{z_n^\tau\}$ starting from $z_0$ with $\phi_1(z)=W(t)+St$. If $z_0\le y_0$ then $z_n^\tau\le y_n^\tau$ for all $n\ge1$ and $\tau>0$.
\end{proposition}

{
\begin{remark}\label{energylevel}
For $\{y_n^\tau\}$ solution of \eqref{rescaledp} we either have $y_n^\tau\le y_{n-1}^\tau$ or $
y_n^\tau\le y_{n-1}^\tau+1$, since $$\phi_1(t)+{a_n^\tau\over 2}\gamma(t-y_{n-1}^\tau)^2/2<\phi_1(t+1)+{a_n^\tau\over 2}\gamma(t+1-y_{n-1}^\tau)^2\hbox{ if }y_{n-1}^\tau\le t\le y_{n-1}^\tau+1.
$$

Moreover, $y_{n+1}\le y_{n-1}^\tau+1$. Indeed, otherwise $y_{n-1}^\tau+1<y_{n+1}^\tau\le y_{n-1}^\tau+2$, by the observation above. As \begin{equation}\label{elineq1}
\phi_1(y_{n+1}^\tau-1)<\phi_1(y_{n+1}^\tau)<\phi_1(y_n^\tau)<\phi_1(t)
\end{equation}
for any $y_{n-1}^\tau<t<y_n^\tau$ and therefore
\begin{equation}\label{elineq2}
y_n^\tau<y_{n+1}^\tau-1\le y_{n-1}^\tau<y_{n+1}^\tau,
\end{equation}
\eqref{elineq1} and \eqref{elineq2} lead to a contradiction. Reasoning by induction we get that $y_{n+k}^\tau\le y_{n-1}^\tau+1$ for any $k\ge0$, $n\ge1$. This means that, even if the motion is not (in general) monotone, once it reaches an energy well either it proceeds further decreasing the energy or it remains in that well.
\end{remark}
}

Consider the case of $N$-periodic perturbations; that is,
\begin{equation}\label{periodicpert}
a^\tau_n=a_i,\text{ if }n=kN+i\text{ for some }k\in\mathbb{N},
\end{equation}
for any $\tau$. The solutions of \eqref{rescaledp} with such perturbations do not depend on $\tau$.

\begin{proposition}\label{homvel}
Consider $\{a_n^\tau\}$ as in \eqref{periodicpert} and $\{y_n\}$ a solution of \eqref{rescaledp} with $\phi_1(t)=W(t)+Tt$. Then there exists the limit
$$f_\gamma(T,\{a_n\})=\lim_{n\to\infty} \frac{y_0-y_n}{n}\ge0$$
and it is independent of $y_0$. Moreover $T\mapsto f_\gamma\big(T,(a_n)\big)$ is an increasing map.
\end{proposition}
\begin{proof}
First, notice that, by the periodicity of $W$, for any integer $l$ the solution of \eqref{rescaledp} from $y_0+l$ is $y_n^l=y_n+l$. Indeed
\begin{align*}
y_1^l &\in \argmin{t\in\mathbb{R}}\bigg\{W(t)+Tt+a_1^\tau\frac{\gamma}{2}\big(t-(y_0+l)\big)^2\bigg\} \\
&=\argmin{t\in\mathbb{R}}\bigg\{W(t-l)+T(t-l)-Tl+a_1^\tau\frac{\gamma}{2}\big((t-h)-y_0\big)^2\bigg\} \\
&= \argmin{s\in\mathbb{R}}\bigg\{W(s)+Ts+a_1^\tau\frac{\gamma}{2}\big(s-y_0\big)^2\bigg\}+l,
\end{align*}
and the claim follows by induction. We first consider the case $y_0=0$. For any $k\ge1$ let $h=\lfloor y_{kN}\rfloor$, so that $0\le y_{kN}-h< 1$. If $z_0=y_{kN}-h$ then we have that $z_i=y_{kN+i}-h$ by the $N$-periodicity of $\{a_n\}$. By Proposition \ref{monbeh} with $T=S$, since $0<z_0\le1$, we have that $y_i\le z_i\le y_i+1$, which reads as $y_i\le y_{kN+i}-h\le y_i+1$, and by the definition of $h$ we have
\begin{equation}\label{ineq}
	y_i+y_{kN}-1\le y_{kN+i}\le y_i+y_{kN}+1.
\end{equation}
For all $m<n$, denote $m'=\lfloor m/N\rfloor N$ and consider $k$ such that and $n=km'+i$, with $0\le i< m'$. Since $m'$ is a multiple of $N$, the second inequality in (\ref{ineq}) applies and we get
$$\frac{y_n}{n}\le\frac{y_{km'}+y_i+1}{km'+i}\le \frac{ky_{m'}+y_i+k}{km'+i}\le\frac{y_{m'}}{m'}+\frac{1}{m'}+\frac{y_i}{km'}.$$
Taking the limit as $n\to\infty$ and keeping $m$ fixed we have
$$\limsup_{n\to\infty}\frac{y_n}{n}\le\frac{y_{m'}+1}{m'}.$$
Now, consider $0\le j<N$ such that $m=\lfloor m/N\rfloor N+j=m'+j$. By the first inequality in \eqref{ineq} we have $y_m\ge y_{m'}+y_j-1$, so that $(y_{m'}+1)/m'\le(y_m+2-y_j)/m'$. Taking the limit as $m\to\infty$ we have that
$$\limsup_{n\to\infty}\frac{y_n}{n}\le\liminf_{m\to\infty}\frac{y_m}{m}$$
which yields the existence of the limit, which is non-negative from Remark \ref{energylevel}. 

Now, consider any $y_0\not=0$ and let $h=\lfloor y_0\rfloor$, so that $0\le y_0-h\le1$. We have
$$y_n^0-1 \le y_n-y_0 \le y_n^0+1,$$
where $\{y_n^0\}$ is the discrete solution of \eqref{rescaledp} starting from $0$. This implies that $f_\gamma\big(T,(a_n)\big)$ is independent of the initial data $y_0$. Moreover, let $\{z_n\}$ be the rescaled discrete solution for linear energy component $S\ge T$. By Proposition \ref{monbeh} we have $-z_n\ge -y_n$, which implies the monotonicity of $f_\gamma$ in $T$.
\end{proof}

If $a_n\equiv c$ then, following the notation in \cite{ansbrazim}, we have $f_\gamma(T,\{a_n\})=f_{c\gamma}(T)$.

\begin{theorem}
Let $\{a_n\}$ be as in \eqref{periodicpert}; then for all $\gamma\in(0,+\infty)$, and for any initial data $u_0\in\mathbb{R}$, the $\{a^\tau\}$-perturbed minimizing movement along $\{\phi_{\gamma\tau}\}$ is
$$u^\gamma(t)=u_0-\gamma f_\gamma\big(T,(a_n)\big)t\quad\hbox{ for all }t\ge0.$$
\end{theorem}
\begin{proof}
For every $t>0$, the pointwise convergence of discrete solutions implies
$$\frac{u^\gamma(t)-u_0}{t} = \lim_k \frac{u^{\tau_k,\varepsilon_k}(t)-u_0^{\varepsilon_k}}{t}.$$
Consider $n=\lceil t/\tau_k\rceil$, which depends on $t$ and $k$. For every $t>0$, we have $n\to\infty$ when $k\to\infty$.  We then obtain
$$\frac{u^\gamma(t)-u_0}{t} = \gamma \lim_k \frac{y_n-y_0}{n}=-\gamma f_\gamma\big(T,(a_n)\big)$$
by multiplying and dividing by $\varepsilon_k$.
\end{proof}

\subsection{The pinning threshold}

We extend the definition of the pinning threshold given in \cite{ansbrazim} (see Definition 5.2) to the perturbed case, still working with periodic perturbations as in \eqref{periodicpert}.

\begin{definition}
For any $\gamma>0$, the {\em pinning threshold} with perturbations $\{a_n\}$ at regime $\gamma$ is
$$T_\gamma(\{a_n\}):=\sup\{T>0\,|\,f_\gamma(T,\{a_n\})=0\}.$$
\end{definition}

Notice that, by the monotonicity of $T\mapsto f_\gamma(T,\{a_n\})$, then $f_\gamma(T,\{a_n\})=0$ for every $T<T_\gamma(\{a_n\})$.

If $a_n\equiv c$, following the notation in \cite{ansbrazim}, we have $T_\gamma(\{a_n\})=T_{c\gamma}$.

{
\begin{proposition}
Given $\{a_n\}$ as in \eqref{periodicpert}, and set $\alpha=\min_{1\le i\le N}a_i$. Then
$$T_\gamma(\{a_n\})=T_{\alpha\gamma}.$$
\end{proposition}
\begin{proof}
Consider $T<T_{\alpha\gamma}$ and denote by $\{y_n^\alpha\}$ a solution of \eqref{rescaledp} starting from $y_0$ corresponding to $a_n\equiv\alpha$. We have that $(y_0-y_n^\alpha)/n\to0$ by the definition of pinning threshold. As already noted, the unperturbed discrete solutions are monotone, and so is $y_n^\alpha$. Assume that $y_1^\alpha\le y_0$. By Lemma \ref{pinninglem}, $y_1^\alpha\le y_1$. Consider
$$z\in\argmin{t\in\mathbb{R}}\left\{\phi_1(t)+a_2\frac{\gamma}{2}(t-y_1^\alpha)^2\right\}.$$
On the one hand, from Proposition \ref{monbeh} we get that $z\le y_2$; on the other hand, by applying Lemma \ref{pinninglem}
again we obtain  $y_2^\alpha\le z$, and by induction we get $y_n^\alpha\le y_n$. Hence, $0\le(y_0-y_n)/n\le(y_0-y_n^\alpha)/n\to 0$; that is, $f_\gamma(T,\{a_n\})=0$. Analogously, if $y_1^\alpha\ge y_0$. This yields that $T_{\alpha\gamma}\le T_\gamma(\{a_n\})$.

Now take $T>T_{\alpha\gamma}$. Since $f_{\alpha\gamma}(T)>0$, $\{y_n^\alpha\}$ is decreasing. It is not restrictive to suppose that $a_1=\alpha$. From Remark \ref{energylevel} we infer that $y_N\le y_1^\alpha+1$, which yields that $y_{N+1}\le y_2^\alpha+1$. Analogously, $y_{2N}\le y_2^\alpha+1$, and therefore $y_{2N+1}\le y_3^\alpha+1$. Hence, we have that $y_{kN+1}\le y_{k+1}^\alpha+1$, which implies that $(y_0-y_{k+1}^\alpha)/(k+1)\le(y_0-y_{kN+1}+1)/(kN+1)$. By taking the limit as $k\to\infty$ we have that $0<f_{\alpha\gamma}(T)\le Nf_\gamma(T,\{a_n\})$, in particular $T>T_\gamma(\{a_n\})$ and the thesis follows.
\end{proof}
}

\section{Perturbed motion of discrete interfaces}

Since the interest for studying the perturbation method defined in Section \ref{pertMM} relies on the competition between an energy and a dissipation term, it can also be applied to the well-know scheme of geometric minimizing movements introduced in the pioneering work \cite{almtaywan}.

In this section, we analyze the effect of the perturbations on the motion of two-dimensional discrete interfaces studied in \cite{bragelnov}.

\subsection{Setting of the problem}

We briefly recall the setting of the problem; for any $\e>0$ we consider the lattice $\e\mathbb{Z}^2$. For any set of indices $\mathcal{I}\subset\varepsilon\mathbb{Z}^2$ we define
$$P_\e(\mathcal{I})=\e\#\{(i,j)\,|\,i\in\mathcal{I},\,j\notin\mathcal{I},\,|i-j|=\e\}.$$
We generalize these functionals to any set $E\subset\mathbb{R}^2$ which is union of squares centered at a point of the lattice $\e\mathbb{Z}^2$ with side length $\e$, that is
\begin{equation}\label{latticeset}
E=\bigcup_{i\in\mathcal{I}}Q_\e(i),
\end{equation}
where $Q_\e(i)=[i_1-\e/2,i_1+\e/2]\times[i_2-\e/2,i_2+\e/2]$, and denote with ${\mathcal  D}_\varepsilon$ the family of such sets. With a slight abuse of notation we write $P_\e(E)=P_\e(\mathcal{I})$ for any set as in \eqref{latticeset}.

If we consider the family of sets of finite perimeter $S=\{E\subset\mathbb{R}^2\,\vert\,\mathcal{H}^1(\partial^*E)<\infty\}$, where $\partial^*E$ denotes the reduced boundary of $E$, endowed with the Hausdorff distance, such functionals correspond to
$$P_\varepsilon(E)=\begin{cases}
\mathcal{H}^1(\partial E) & E\in {\mathcal  D}_\varepsilon\\
+\infty&\text{otherwise,}
\end{cases}$$
defined in $S$, with domain ${\mathcal  D}(P_\e)={\mathcal  D}_\e$. As already noted in \cite{bragelnov} and proved for instance in \cite{alibracic}, 
the functionals $P_\e$ approximate (in the sense of $\Gamma$-convergence) the crystalline perimeter
$$P(E)=\int_{\partial^*E}\Vert \nu\Vert_1 d\mathcal{H}^1.$$

\begin{remark}
These functionals could also be seen as interface energies on spin systems. Given $u:\mathbb{Z}^2\to\{\pm1\}$ one defines the interaction energies
$$E_\e(u)=\frac{\e}{4}\sum_{\substack{i,j\in\mathcal{I}\\ |i-j|=1}}|u_i-u_j|^2$$
and $E_\e(u)=P_\e(\{i\in\e\mathbb{Z}^2\,|\,u(i/\e)=1\})$.
\end{remark}

The dissipations $D_\e(F,E)$ are defined as follows (according to the notation in \cite{bragelnov}). For any $\mathcal{I}\subset\e\mathbb{Z}^2$ and $E$ as in \eqref{latticeset}, we define the {\em discrete distance from $\partial E$} 
of $x\in Q_\e(i)$ for some $i\in\e\mathbb{Z}^2$ as
$$
d_\infty^\e(x,\partial E)=\begin{cases}{\e\over2}+
\text{dist}_\infty(i,\mathcal{I})&i\notin\mathcal{I} \\{\e\over2}+
\text{dist}_\infty(i,\e\mathbb{Z}^2\backslash\mathcal{I})&i\in\mathcal{I}
\end{cases}
$$
(the term $\e/2$ comes from the fact that the minimal distance of a point in $\e\mathbb{Z}^2$ from $\partial E$ is  $\e/2$, so that
in this way $d_\infty^\e(x,\partial E)\in\e\mathbb{Z}$).
Then, for any $E,F\in {\mathcal  D}_\varepsilon$
$$D_\varepsilon(F,E)=\int_{E\triangle F}d_\infty^\e(x,\partial E)dx,$$
where $E\triangle F$ denotes the symmetric difference between $E$ and $F$.

Finally, we can set the following minimization scheme
\begin{equation}\label{discretemot}
\begin{cases}
E_0^{\tau,\e}\in S \\ \displaystyle
E_n^{\tau,\e}\in\argmin{E\in S}\Big\{P_\e(E)+{a_n^\tau\over\tau} D_\e(E,E_{n-1}^{\tau,\e})\Big\},
\end{cases}
\end{equation}
and again denote as $E^{\tau,\e}(t)=E^{\tau,\e}_{\lceil t/\tau\rceil}$ its discrete solutions. Any limit in the Hausdorff metric of $\{E^{\tau,\e}\}$ is a \emph{$\{a^\tau\}$-perturbed geometric minimizing movement at regime $\tau$-$\e$}.

In the following, we study the perturbed scheme \eqref{discretemot} in the special case in which $E_0^{\tau,\e}=E_0^\e$ are coordinate rectangles converging in the Hausdorff metric as $\e\to0$, because this case already provides an interesting example of the effects of the perturbation. We also consider $\{a^\tau\}$ satisfying assumption 5 of Remark \ref{epmm}.

We focus our analysis on the regimes $\e=\gamma\tau$, which has been proved to be the critical one for the un-perturbed case.
The fast-converging cases are obtained as limit cases of the critical ones, cf.~Remark \ref{relica}.

\subsection{Motions of coordinate rectangles}

In \cite{bragelnov} Theorem 1 it has been proved that motions $\{E^{\tau,\e}\}$ starting from a coordinate rectangle $E^\e$ keep the rectangular shape. We confine the analysis to the evolution of the lengths $L_{1,n}^{\tau,\varepsilon}$, $L_{2,n}^{\tau,\varepsilon}$ of the sides of $E^{\tau,\e}_n$.

In the proof of Theorem 1 \cite{bragelnov} it is proved that, if $E$ is a coordinate rectangle in ${\mathcal  D}_\e$ of sides of length $L_1, L_2$, the minimizer
$$E'\in\argmin{F\in {\mathcal  D}_\e}\{P_\e(F)+\eta D_\e(F,E)/\e\}$$
is a coordinate rectangle centered in the center of $E$ and with sides of length $L_1', L_2'$ which satisfies 
$$\frac{L'_1-L_1}{\e}=-2\left\lfloor\frac{2}{\eta L_2}\right\rfloor+O(\e^2),\quad \frac{L'_2-L_2}{\e}=-2\left\lfloor\frac{2}{\eta L_1}\right\rfloor+O(\e^2),$$
except when $2/\eta L_1$ or $2/\eta L_1$ is in a neighborhood of an integer of amplitude which is infinitesimal with respect to $\e$. {
In order to simplify the exposition, we omit these cases, as their treatment does not 
vary from that of the corresponding cases in \cite{bragelnov}.} By taking $\eta=\gamma a_n^\tau$, and $E=E^{\tau,\e}_{n-1}$, for any $n\ge 1$ we have
\begin{equation}\label{increments}
\begin{aligned}
\frac{L_{1,n}^{\tau,\varepsilon}-L_{1,n-1}^{\tau,\varepsilon}}{\tau} &=-2\gamma\left\lfloor\frac{2}{\gamma a_n^\tau L_{2,n-1}^{\tau,\varepsilon}}\right\rfloor+O(\e^2) \\
\frac{L_{2,n}^{\tau,\varepsilon}-L_{2,n-1}^{\tau,\varepsilon}}{\tau} &=-2\gamma\left\lfloor\frac{2}{\gamma a_n^\tau L_{1,n-1}^{\tau,\varepsilon}}\right\rfloor+O(\e^2).
\end{aligned}
\end{equation}
Note that, as in Remark \ref{epmm}, equation \eqref{regularity} reads
$$D_\e(E^{\tau,\e}(t),E^{\tau,\e}(s))\le c\,\theta_T(t+\tau,s)$$
for any $t,s\in[0,T]$. This means that (up to subsequences) $L_1^{\tau,\e}, L_2^{\tau,\e}$ converge pointwise to some absolutely continuous functions $L_1, L_2$, as $\tau\to0$. This implies the pointwise convergence, in the Hausdorff metric, of the discrete solutions; i.e., $d_{\mathcal{H}}(E^{\tau_k,\e_k}(t),E(t))\to 0$ for some infinitesimal $\{\tau_k\}, \{\e_k\}$, to the coordinate rectangle $E(t)$ with sides of length $L_1(t), L_2(t)$. In particular $E(t)$ satisfies the following system of coupled ordinary differential equations
\begin{equation}\label{syst}
\begin{cases}
L'_{1}(t)=-2\gamma v_{1}(t) \\
L'_{2}(t)=-2\gamma v_{2}(t),
\end{cases}
\end{equation}
where $v_1$ and $v_2$ are, respectively, the weak limit in $L^1(0,T)$ of $\lfloor 2/(\gamma a^\tau(t)L_2^{\tau,\e}(t))\rfloor$ and $\lfloor 2/(\gamma a^\tau(t)L_1^{\tau,\e}(t))\rfloor$.

\subsection{The case of periodic perturbations} We can find the explicit form of $v_1$ and $v_2$, for instance in the case of $N$-periodic perturbations $\{a_n\}$ as in \eqref{periodicpert}. For any $1\le i\le N$ consider the functions
$$\chi^\tau_i(s):=\begin{cases}
1 & s\in((kN+i-1)\tau,(kN+i)\tau],\, k\in\mathbb{N} \\
0 & \text{otherwise,}
\end{cases}$$
which weakly converge to $1/N$ in $L^1_{\rm loc}$. We can rewrite the first of \eqref{increments} as
\begin{align*}
L^{\tau,\varepsilon}_{1,n} &= L^{\tau,\varepsilon}_{1,n-1}-2\gamma\tau \left\lfloor\frac{2}{\gamma a_n}\frac{1}{L^{\tau,\varepsilon}_{2,n}}\right\rfloor = L^{\tau,\varepsilon}_{1,0}-2\gamma\tau\sum_{k=0}^n \left\lfloor\frac{2}{\gamma a_k}\frac{1}{L^{\tau,\varepsilon}_{2,k}}\right\rfloor \\
&= L^{\tau,\varepsilon}_{1,0} - 2\gamma\tau \sum_{i=1}^N \sum_{k=1}^{n/N} \left\lfloor\frac{2}{\gamma a_i}\frac{1}{L^{\tau,\varepsilon}_{2,kN+i}}\right\rfloor \\
&= L^{\tau,\varepsilon}_{1,0}-2\gamma \sum_{i=1}^N \int_0^{n\tau} \left\lfloor\frac{2}{\gamma a_i}\frac{1}{L^{\tau,\varepsilon}_{2}(s)}\right\rfloor\chi^\tau_i(s)ds\,,
\end{align*}
and, taking the limit as $\tau\to0$, by the weak convergence of $\chi^\tau_i$ and the local uniform convergence of $L^{\tau,\varepsilon}_{2}$, we obtain that $v_1=\sum_{i=1}^N \lfloor 2/(\gamma a^\tau_i L_2)\rfloor/N$. We argue analogously for $v_2$; hence, \eqref{syst} reads as
\begin{equation}\label{systper}
\begin{cases}\displaystyle
L'_{1}(t)=-2\gamma{1\over N}\sum_{i=1}^N \Biggl\lfloor {2\over\gamma a^\tau_i L_2(t)}\Biggr\rfloor
\\  \\ \displaystyle
L'_{2}(t)=-2\gamma{1\over N}\sum_{i=1}^N \Biggl\lfloor {2\over\gamma a^\tau_i L_1(t)}\Biggr\rfloor.
\end{cases}
\end{equation}

Note that in the perturbed case the {\em pinning condition} changes; indeed, defining $\alpha=\min_{1\le i\le N}a_i$, if
$$L_{1,0}>\frac{2}{\gamma\alpha},$$
then $L'_2(t)=0$ for $t>0$ small enough. The same holds for $L_1$. As in \cite{bragelnov} Theorem 2 the motion is characterized in three cases:
\begin{itemize}
\item[(i)] \emph{total pinning}; i.e., $E(t)\equiv E_0$, if $L_{1,0}, L_{2,0}>2/(\gamma\alpha)$;
\item[(ii)] if $L_{1,0}\le2/(\gamma\alpha)$ and $L_{2,0}<2/(\gamma\alpha)$, then $E(t)$ shrinks up to an extinction time;
\item[(iii)] \emph{partial pinning} if $L_{1,0}>2/(\gamma\alpha)$ and $L_{2,0}<2/(\gamma\alpha)$ with $2/(\gamma a_i L_{2,0})\not\in\mathbb{N}$ for any $i$, in which only one couple of sides moves up to a time $t_0$ such that $L_{1,0}(t_0)\le2/(\gamma\alpha)$, then $E(t)$ follows the motion described in the previous case.
\end{itemize}

\begin{remark}[Limit cases]\label{relica}
As noted above, in the un-perturbed case, when $\tau(\e)=o(\e)$ and $\e(\tau)=o(\tau)$ we have fast convergences. In the general case of perturbed motion of discrete interfaces, the situation is slightly more complicated and can be treated separately. Nevertheless, in the periodic case we can make some interesting observations.

When $\tau(\e)=o(\e)$ we have total pinning of the motion, indeed from \eqref{increments}
$$L_{i,1}^{\tau(\e),\e}=L_{0,1}^\e-2\e\left\lfloor\frac{2\tau}{\e a_i L_{2,0}^\e}\right\rfloor=L_{0,1}^\e$$
and the same holds for $L_2^\e$, for $\e$ sufficiently small.

The case $\e(\tau)=o(\tau)$, also for non-periodic perturbations, reduces to the study of the $\Gamma$-limit of the functionals $P_\e/a^{\tau(\e)}$. We do not provide a limit result in this work, but one can observe that taking the limit in \eqref{systper} as $\gamma\to0$ we get
$$\begin{cases}\displaystyle
L_1'(t)=-{2\over a^*L_2(t)} \\ \displaystyle
L_2'(t)=-{2\over a^*L_1(t)},
\end{cases}$$
where $a^*$ is the harmonic mean of $\{a_n\}$. This can be regarded as the \emph{flat flow} perturbed by $a^*$, with respect to the crystalline perimeter $P$ defined above (see \cite{almtay} for the study of the flat flow  in the un-perturbed case;
see also \cite{bra2} Section 9.4).
\end{remark}

\section*{Acknowledgments} The authors acknowledge the MIUR Excellence Department Project awarded to the Department of Mathematics, University of Rome Tor Vergata, CUP E83C18000100006.



\begin{thebibliography}{11}

\bibitem{alibracic}
	\newblock R. Alicandro, A. Braides and M. Cicalese,
	\newblock Phase and anti-phase boundaries in binary discrete systems: a variational viewpoint,
	\newblock \emph{Netw. Heterog. Media}, \textbf{1} (2006), 85-107.

\bibitem{almtay}
	\newblock F. Almgren and J. L. Taylor,
	\newblock Flat flow is motion by mean curvature for curves with crystalline energy,
	\newblock \emph{J. Differential Geom.}, \textbf{42} (1995), 1--22.

\bibitem{almtaywan}
	\newblock F. Almgren, J. L. Taylor and L. Wang,
	\newblock Curvature-driven flows: a variational approach,
	\newblock \emph{SIAM J. Control Optim.}, \textbf{31} (1993), 387-438.

\bibitem{ambgigsav}
	\newblock L. Ambrosio, N. Gigli and G. Savar\'e,
	\newblock \emph{Gradient Flows in Metric Spaces and in the Space of Probability Measures}, $2^{nd}$ edition,
	\newblock Birkh\"auser, Basel, 2008.

\bibitem{ansbrazim}
	\newblock N. Ansini, A. Braides and J. Zimmer,
	\newblock Minimizing movements for oscillating energies: the critical regime,
	\newblock \emph{\it  Proc. Royal Soc. Edin. A}, {\bf 149} (2019), 719--737

\bibitem{bra1}
	\newblock A. Braides,
	\newblock \emph{$\Gamma$-convergence for Beginners}
	\newblock Oxford University Press, Oxford, 2002.

\bibitem{bra2}
	\newblock A. Braides,
	\newblock \emph{Local Minimization, Variational Evolution and $\Gamma$-convergence},
	\newblock Springer, Cham, 2014.

\bibitem{bracolgobsol}
	\newblock A. Braides, M. Colombo, M. Gobbino and M. Solci,
	\newblock Minimizing movements along a sequence of functionals and curves of maximal slope,
	\newblock \emph{C. R. Math. Acad. Sci. Paris}, \textbf{354} (2005), 685-689.

\bibitem{bragelnov}
	\newblock A. Braides, M.S. Gelli and M. Novaga,
	\newblock Motion and pinning of discrete interfaces,
	\newblock \emph{Arch. Ration. Mech. Anal.}, \textbf{195} (2010), 469-498.
	
\bibitem{deg}
	\newblock E. De Giorgi,
	\newblock New problems on minimizing movements,
	\newblock in \emph{Boundary Value Problems for Partial Differential Equations and Applications} (C. Baiocchi and J. L. Lions, eds.) Masson, Paris, 1993, pp. 81-98.
	
\bibitem{donfremie}
	\newblock P. Dondl, T. Frenzel and A. Mielke,
	\newblock A gradient system with a wiggly energy and relaxed EDP-convergence,
	\newblock \emph{ESAIM Control Optim. Calc. Var.},  to appear.

\bibitem{fle}
     \newblock F. Fleissner,
     \newblock $\Gamma$-convergence and relaxation of gradient flows in metric spaces: a minimizing movement approach,
     \newblock \emph{ESAIM Control Optim. Calc. Var.} {\bf 25} (2019), Art. 28

\bibitem{flesav}
    \newblock F. Fleissner and G. Savar\'e,
    \newblock Reverse approximation of gradient flows as minimizing movements: a conjecture by De Giorgi,
	\newblock preprint, \arXiv{1711.07256}.
	
	\bibitem{SS}
    \newblock 	E. Sandier and S. Serfaty.
    \newblock $\Gamma$-convergence of gradient flows with applications to Ginzburg-Landau.     
    \newblock {\it Commun. Pure Appl. Math.} {\bf 57} (2004), 1627--1672.

\bibitem{tri}
	\newblock A. Tribuzio,
	\newblock Perturbations of minimizing movements and curves of maximal slope,
	\newblock \emph{Netw. Heterog. Media}, \textbf{13} (2018), 423-448.

\end{thebibliography}
\end{document}